\documentclass{amsart}
\usepackage{graphicx}
\usepackage{color}
\usepackage{ifpdf}
\usepackage{amsmath,amssymb,amsfonts}
\newtheorem{theorem}{Theorem}[section]
\newtheorem{lemma}[theorem]{Lemma}

\theoremstyle{definition}
\newtheorem{definition}[theorem]{Definition}

\newtheorem{example}[theorem]{Example}
\newtheorem{cor}[theorem]{Corollary}

\numberwithin{figure}{section}
\newtheorem{remark}[theorem]{Remark}
\numberwithin{equation}{section}

\def \A{\mathcal{A}}
\def \D{\mathcal{D}}
\def \g{\gamma}
\def \G{\Gamma}

\def \im{\mathbf{i}}
\def \I{\mathcal{I}}

\def \lamax{\lambda_{\max}}
\def \L{\mathcal{L}}

\def \rmI{\mathrm{I}}
\def \s{\sigma}

\DeclareMathOperator{\sgn}{sgn}
\DeclareMathOperator{\diag}{diag}

\begin{document}
\title[Spectral method to incidence balance of oriented hypergraphs]
{A spectral method to incidence balance of oriented hypergraphs and induced signed hypergraphs}

\author[Y. Wang]{Yi Wang}
\address{School of Mathematical Sciences, Anhui University, Hefei 230601, P. R. China}
\email{wangy@ahu.edu.cn}

\author[L. Wang]{Le Wang}
\address{School of Mathematical Sciences, Anhui University, Hefei 230601, P. R. China}
\email{wangl@stu.ahu.edu.cn}

\author[Y.-Z. Fan]{Yi-Zheng Fan*}
\address{School of Mathematical Sciences, Anhui University, Hefei 230601, P. R. China}
\email{fanyz@ahu.edu.cn}
\thanks{*The corresponding author. This work was supported by National Natural Science Foundation of China (Grant No. 11771016, 11871073)}

\begin{abstract}
An oriented hypergraph is a hypergraph together with an incidence orientation such that each edge-vertex incidence is given a label of $+1$ or $-1$.
An oriented hypergraph is called incidence balanced if there exists a bipartition of the vertex set
  such that every edge intersects one part of the bipartition in positively incident vertices with the edge
  and other part in negatively incident vertices with the edge.
In this paper, we investigate the incidence balance of oriented hypergraphs and induced signed hypergraphs by means of the eigenvalues of associated matrices and tensors,
and provide a spectral method to characterize the incidence balanced oriented hypergraphs and their induced signed hypergraphs.
\end{abstract}



\maketitle

\section{Introduction}

An {\it oriented hypergraph} is a hypergraph together with an oriented structure such that each edge-vertex incidence is given a label of $+1$ or $-1$ \cite{RR, Rusnak}.
The sign of the edges of an oriented hypergraphs is naturally induced by the oriented edge-vertex incidence structure \cite{YYQ}.
An oriented hypergraph is sometime called a {\it signed hypergraph}
  if researchers focus on the sign of edges rather than the edge-vertex incidence orientation \cite{Shi, SB, YYQ}.
Clearly, a {\it signed graph} can be regarded as an oriented hypergraph in which each edge contains exactly two incidences \cite{Z2}.
So, the research of oriented hypergraphs extends that of signed graphs, among of which the matrices are used to investigate the oriented hypergraphs;
see \cite{CRRY, Reff1, Reff2, RR, RS}.

Recently, the spectral hypergraph theory has been an active topic in algebraic graph theory; see e.g. \cite{CD,FBH,FHB,LM,Ni,PZ,Qi}.
The adjacency tensors \cite{CD} and Laplacian tensors \cite{Qi} of uniform hypergraphs were introduced to investigate the structure of hypergraphs,
  just like adjacency matrices and Laplacian matrices to simple graphs.
In 2019, Yu et al. \cite{YYQ} extended some notions of the adjacency and Laplacian tensor from uniform hypergraphs to uniform oriented hypergraphs,
and presented some spectral properties of the adjacency or Laplacian tensors of uniform oriented hypergraphs.

A signed graph is said to be {\it balanced} if the product of the signs of all edges in every cycle is positive.
The balance is a global structural property of signed graphs, which is critical to many problems on combinatorial optimization and programming \cite{Sch}.
The signed graphs and their balance can be traced back to 1950s, which were introduced by Harary \cite{Harary} to treat a question in social psychology \cite{CH}.
Harary's Balance Theorem \cite{Harary}
tells us that a signed graph is balanced if and only if there exists a (possibly trivial) bipartition $\{X,Y\}$ of the vertex set
 such that an edge is negative precisely when it has one endpoint in $X$ and the other in $Y$.
The other characterizations for the balance of a signed graph were investigated from the structural or algebraic aspects \cite{Harary, HK, HLP,Z1}.
The notion of balance was generalized to oriented hypergraphs by Shi and Brzozowski \cite{SB} to model VLSI optimization problems.

\begin{definition}\cite{SB}\label{bdef}
An oriented hypergraph is called ({\it incidence}) {\it balanced} if there exists a bipartition of its vertices that balances
all of its edges, namely, every edge intersects one part of the bipartition in positively incident vertices with the edge
  and other part in negatively incident vertices with the edge.
\end{definition}

It is shown that an oriented hypergraph is balanced
if and only if the product of the signs of all incidences contained in each cycle is positive \cite{SB},
  which can be viewed as an {\it incidence version} of the balance of hypergraphs.
Rusnak \cite{Rusnak} gives another version of balance from the adjacency viewpoint,
  where an oriented hypergraph is called balanced if the product of the signs of all adjacencies contained in each cycle is positive.
 To avoid confusions, we refer to the balance defined in Definition \ref{bdef} as {\it incidence balance}, and the balance defined in \cite{Rusnak} as {\it adjacency balance}. In past years, researchers have characterized the adjacency balance of oriented hypergraphs by means of their structure \cite{Rusnak} and associated matrices \cite{CRRY}.

In this paper, we focus on the incidence balance of oriented hypergraphs
  and attempt to characterize the incidence balanced oriented hypergraphs in terms of the eigenvalues of their associated matrices and tensors.
The rest of this paper is organized as follows.
In Section 2, some basic notions together with their properties are introduced,
as well as some equivalent conditions for incidence balanced oriented hypergraphs are presented from the structural aspect.
In Section 3, our main result provides a spectral characterization of incidence balanced oriented hypergraphs by using the adjacency or Laplacian matrices.
In the final section, we characterize the signed hypergraphs induced by the incidence balanced oriented hypergraphs
  whose underlying hypergraph is connected and $k$-uniform with even $k$, by means of the eigenvalues of their adjacency or Laplacian tensors.

\section{Basic notions}
\subsection{Signed hypergraphs and oriented hypergraphs }

For a positive integer $n$, denote a set $[n]:=\{1, \ldots, n\}$.
A {\it hypergraph} $G=(V(G),E(G))$ is a pair consisting of a vertex set $V(G)=\{v_1,\ldots,v_n\}$ and an edge set $E(G)=\{e_{1},\ldots,e_m\}$,
  where $e_j\subseteq V(G)$ for each $j\in[m]$.
If $|e_j|=k$ for all $j\in[m]$, then $G$ is called a {\it $k$-uniform} hypergraph.
A {\it signed hypergraph} $\G=(G,\g)$ is a hypergraph $G$ together with a sign $\g$ of its edges, i.e. $\g:E(G) \to \{-1,1\}$,
  where an edge $e$ of $\G$ is called {\it positive} (or {\it negative}) if its sign $\g(e)=1$ (or $-1$).

An order pair $(e,v)$ is called an {\it incidence} if $v \in e$.
The {\it degree} of a vertex $v$, denoted by $d_v$, is defined to be the number of incidences containing $v$.
 Two vertices $u$ and $v$ are said to be {\it adjacent} with respect to an edge $e$, if there exist incidences $(e,u)$ and $(e,v)$.
An {\it adjacency} is a triple $(u,v; e)$ if $u$ and $v$ are adjacent with respect to the edge $e$.

A {\it walk} is a sequence $P=a_0, i_1, a_1, i_2, \ldots, a_{t-1}, i_t, a_t$, where $\{a_k\}_0^t$ is an alternating sequence of vertices and edges, $i_j$ is an incidence consisting of $a_{j-1}$ and $a_j$ for $j \in [t]$.
The walk $P$ is called a {\it path} if no vertices, edges, or incidences appeared in $P$ are repeated,
and is called a {\it cycle} if no vertices, edges, or incidences appeared in $P$ are repeated, expect $a_0=a_t$.
By symmetry we may assume a cycle starts from a vertex.
A hypergraph $G$ is {\it connected} if any two distinct elements of $V(G) \cup E(G) $ are connected by a walk in $G$.

\begin{figure}
\centering
  \setlength{\unitlength}{1bp}%
  \begin{picture}(200, 64.74)(0,0)
  \put(0,0){\includegraphics[scale=.8]{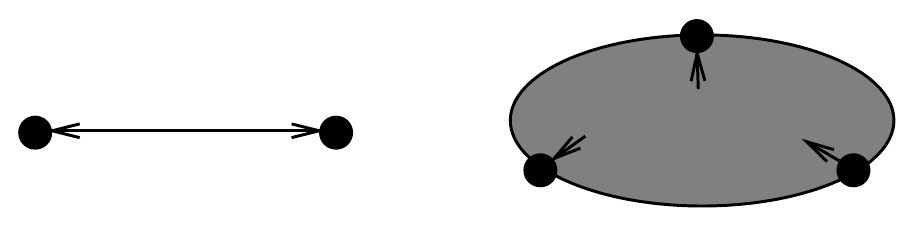}}
  \end{picture}%
\caption{Two oriented edges}\label{ori}
\end{figure}

Let $\I(G)$ denote the set of all incidences of $G$.
An {\it incidence orientation} of $G$ is a map $\sigma:  \mathcal{I}(G) \to \{-1,1\}$.
The hypergraph $G$ together with an incidence orientation $\sigma$ is called an {\it oriented hypergraph}, denoted by $G^\sigma$ or $(G,\sigma)$.
If we draw edges of $G^{\sigma}$ as shaded regions in the plane
whose incident vertices are arranged as points on its boundary, an oriented incidence $(e,v)$ can be viewed as an arrow existing $v$ when $\sigma(e,v)=1$,
or an arrow entering $v$ when $\sigma(e,v)=-1$; see Fig. \ref{ori}.
It should be noted that the signed graphs discussed in \cite{Z2} can be regarded as a special case of oriented $k$-uniform hypergraphs when $k=2$.

The {\it sign of an adjacency} $(u,v;e)$ in $G^{\s}$ is defined as
$$\sgn_\s (u,v;e)=-\s(e,u) \s(e,v).$$
The {\it sign of an edge} $e$ is defined as
$$\sgn_\s e=(-1)^{|e|-1} \s^e,$$
 where $ \s^e:=\prod_{v \in e} \s(e,v)$.
Let $P$ be path or a cycle of $G^\s$ as follows:
\begin{equation}\label{path}
P=a_0, i_1, a_1, i_2, \ldots, a_{t-1}, i_{t}, a_{t}.
\end{equation}
The {\it adjacency sign} of $P$ is defined as
$$\sigma_A(P):=(-1)^{\lfloor \frac{t}{2} \rfloor} \prod_{j=1}^{t} \sigma(i_j),$$
which implies that the adjacency sign of a cycle is the product of the signs of all adjacencies contained in the cycle.
Similarly, the {\it incidence sign} of $P$ is defined as $$\sigma_I(P):=\prod_{j=1}^{t} \sigma(i_j),$$
which is the product of the signs of all incidences contained in $P$.

For an oriented hypergraph $G^\s$, if we pay more attention to the sign of edges rather than the edge-vertex incidence orientation,
the oriented hypergraph $G^\s$ naturally induces a signed hypergraph $\G=(G,\g)$, denoted by $\G G^\sigma$, by defining $\g(e):=\sgn_\s e$ for each edge $e \in E(G)$.

\subsection{Switching equivalence}
Let $G^{\s}$ be an oriented hypergraph. The {\it edge-vertex incidence matrix} of $G^\s$, denoted by $M(G^\s)$, is defined by $M(G^\s)_{e,v}=\s(e,v)$ if $v \in e$ and
$M(G^\s)_{e,v}=0$ otherwise.
 A {\it vertex switching} $S_v$ on $G^\s$ at a vertex $v$  is an operation on $G^\s$ which reverses the incidence orientation $\s(e,v)$ by $-\s(e,v)$ for all edges $e$
incident to $v$, and keep the other incidence orientation invariant.
Equivalently, if we consider $S_v$ as a map $S_v: V(G) \to \{-1,1\}$ such that $S_v(u)=-1$ if $u =v$ and $S_v(u)=1$ otherwise,
 then $\s S_v$ is the incidence orientation of the resulting hypergraph by defining $\s S_v(e,u):=\s(e,u)S_v(u)$ for each incidence $(e,u)$.
From the viewpoint of the matrix, $S_v$ can be considered as a {\it signature matrix},
    a diagonal matrix whose diagonal entries are taken from $\{1, -1\}$, such that $(S_v)_{uu}=S_v(u)$ for all $u \in V(G)$.
Then
$M(G^{\s S_v})=M(G^\s)S_v$.
Generally, for a subset $U$ of $V(G)$, define a signature matrix $S_U:=\prod_{v\in U}S_v$, which is a sequence of vertex switchings taken at the vertices of $U$.

An {\it edge switching} $S_e$ at an edge $e$  is an operation on $G^\s$ which reverses the incidence orientation $\s(e,v)$ by $-\s(e,v)$ for all vertices $v$ incident with $e$,
and keep the other incidence orientation invariant.
If we consider $S_e$ as a map $S_e: E(G) \to \{-1,1\}$ such that $S_e(f)=-1$ if $f =e$ and $S_e(f)=1$ otherwise,
 then $\s S_e$ is the incidence orientation of the resulting hypergraph by defining $\s S_e(f,u):=\s(f,u)S_e(f)$ for each incidence $(f,u)$.
Similarly, $S_e$ can be considered as a signature matrix such that $(S_e)_{ff}=S_e(f)$ for all $f \in E(G)$, and $M(G^{\s S_e})=S_e M(G^\s)$.
For a subset $F$ of $E$, define $S_F:=\prod_{e\in F}S_e$, which is a sequence of edge switchings taken at the edges of $F$.
The vertex switching and edge switching were introduced by Rusnak \cite{Rusnak} to discuss the adjacency balance.

Two oriented hypergraphs $G^\s$ and $G^{\tau}$ are called {\it switching equivalent}
 if one can be obtained from another by a sequence of vertex switchings and/or edge switchings.
By the above discussion, we immediately obtain the following lemma.

\begin{lemma}\label{veswt}
Two oriented hypergraphs $G^\s$ and $G^\tau$ are switching equivalent if and only if there exist signature matrices $S_F$ and $S_U$
such that $M(G^\s)=S_F M(G^\tau) S_U$ for some subset $F \subseteq E(G)$ and $U \subseteq V(G)$.
\end{lemma}

Let $\G=(G, \g)$ be a signed hypergraph (or a signed graph).
A {\it vertex switching} $S_v$ on $\G$ at a vertex $v$  is an operation on $\G$ which negatives the sign of each edge $e$ incident to $v$,
  and keep the sign of other edges invariant.
Two signed hypergraphs  are called {\it switching equivalent} if one can be obtained from another by a sequence of vertex switchings.

\subsection{The incidence balance of oriented hypergraphs}
The (oriented, signed) {\it incidence graph} of an oriented hypergraph $(G, \s)$, denoted by $\rmI G^{\sigma}$,
is the (oriented, signed) bipartite graph with vertex set $V(\rmI G^\s)=V(G) \cup E(G)$, edge set $E(\rmI G^\s)=\I(G)$ and (orientation, sign) $\sigma$.
When $\rmI G^{\sigma}$ is regarded as a signed graph, an incidence $(e,v)$ in $G^\s$ will become an edge or an adjacency $(e,v)$ in $\rmI G^{\sigma}$ whose sign is $\s(e,v)$.
Furthermore, a path or a cycle $P$ of $G^\s$ as in (\ref{path}) is also a path or cycle of $\rmI G^\s$ by thinking incidence $i_j$ in $G^\s$ as adjacency in $\rmI G^\s$,
 and vice versa.
So the paths or cycles of $G^\s$ are bijectively corresponding to the those of $\rmI G^\s$,
  whose signs are exactly the incidence signs of the corresponding paths or cycles in $G^\s$.
Denote by $G^+$ an oriented hypergraph $G^\s$ with $\s > 0$.
By the above discussion, we immediately get the following result, in which some partial equivalent statements can be found in \cite{SB}.

\begin{theorem} \label{eqs}
Let $G^\s$ be an oriented hypergraph. Then the following are equivalent.

\begin{enumerate}

\item $G^\s$ is incidence balanced.

\item  Each cycle of $G^\s$ has a positive incidence sign.

\item  All $ab$ paths of $G^\s$ have the same incidence sign for any $a, b$ in $V(G) \cup E(G)$.

\item  The signed incidence graph $\rmI G^\s$ is balanced.

\item  $G^\s$ is switching equivalent to $G^+$.
\end{enumerate}
\end{theorem}

\begin{proof}
The equivalence of (1) and (2) can be found in \cite[Theorem 3.1]{SB}.

By the Harary's Balance Theorem, a signed graph is balanced if and only if each cycles has a positive sign, which is also equivalent that
all $vw$ paths have the same sign for any two given vertices $v$ and $w$.
By the correspondence of the cycles or paths in $\rmI G^\s$ and those in $G^\s$ discussed prior this lemma, we get the equivalence of (2) and (4), and also (3) and (4).

 It was also known that a signed graph is balanced if and only if it is switching equivalent to a signed graph with all positive signature \cite[Corollary 3.3]{Z1}.
 Note that a vertex switching on $\rmI G^\s$ is corresponding to a vertex switching or an edge switching on $G^\s$.
 So, $\rmI G^\s$ is switching to $\rmI G^+$ if and only if $G^\s$ is switching to $G^+$, which yields the equivalence of (4) and (5).
 \end{proof}

\section{Matrix for oriented hypergraphs}
For a real matrix $M$ of size $m \times n$, there exist two orthogonal matrices, $U$ of order $m$ and $W$ of order $n$,
  such that $M=U \Lambda W^\top$, where $\Lambda$ is a rectangular diagonal matrix with nonnegative real numbers $\Lambda_{ii}=:\lambda_i$ for $i=1,\ldots, \min\{m,n\}$.
The numbers $\lambda_i$ are called  the {\it singular values} of $M$.
We note that $\lambda_i^2$ are the eigenvalues of $MM^\top$ or $M^\top M$,
  and the columns of $U$ (columns of $W$, respectively) are eigenvectors of $MM^\top$ ($M^\top M$, respectively)
  corresponding to the eigenvalues $\lambda_i^2$ or zero eigenvalues.
Denote by $\lamax(M)$ the maximum singular value of $M$.

The {\it adjacency matrix} of $G^\s$ is defined to be $A(G^\s)=(a^\s_{uv})$,
  where $a^\s_{uv}=\sum_{u,v \in e}\s(e,u)\s(e,v)$ if $u,v$ are distinct and adjacent in $G$, and $a^\s_{uv}=0$ otherwise.
The {\it Laplacian matrix} of $G^\s$ is defined to be $L(G^\s)=D(G)+A(G^\s)$, where $D(G)=\diag \{d_v: v \in V(G)\}$, the  {\it degree matrix} of $G$.

We note that $A(G^+)$ is exactly the adjacency matrix of the hypergraph $G$ introduced by Feng and Li \cite{FL} to study the the second largest eigenvalue.
If $G$ is connected, then $A(G^+)$, as well as $L(G^+)$, is nonnegative and irreducible.
There are also other versions of adjacency matrix and Laplacian matrix of (oriented) hypergraphs; see \cite{CT,Rod,RR}.

Observe that both $A(G^\s)$ and $L(G^\s)$ are real symmetric matrices.
It is also easy to show that
$$ L(G^\s)=M(G^\s)^\top M(G^\s),$$
implying that $L(G^\s)$ is positive semidefinite.
The eigenvalues of $L(G^\s)$ is the square of the singular values of $M(G^\s)$ or zeros.
Denote by $\rho(A)$ the spectral radius of a square matrix $A$.

\begin{lemma}\label{ins} Let $G$ be a connected hypergraph.
Oriented hypergraph $G^\s$ is switching equivalent to $G^+$ if and only if $\lamax(M(G^+))$ is a singular value of $M(G^\s)$.
\end{lemma}

\begin{proof}
Necessity.
If $G^\s$ is switching equivalent to $G^+$, it follows from Lemma \ref{veswt} that $M(G^\s)=S_F M(G^+) S_Z$ for some $F \subseteq E(G)$ and $Z \subseteq V(G)$.
There exist two orthogonal matrices $U$ and $W$, such that
$M(G^+)=U \Lambda W^\top$, where $\Delta_{11}=\lamax(M(G^+))$.
So
$M(G^\s)=S_F U \Lambda W^\top S_Z$.
As $S_F U$ and $S_Z W$ are also orthogonal, $\lamax(M(G^+))$ is a singular value of $M(G^\s)$.

Sufficiency.
Let $\lamax(M(G^+))=:\rho$.
Then $\rho^2$ is the largest eigenvalue (also spectral radius) of $M(G^+)M(G^+)^\top$.
By the assumption,
$\rho$ is a singular value of $G^\s$, and hence $\rho^2$ is an eigenvalue of $M(G^\s)M(G^\s)^\top$ corresponding to a real eigenvector $x$.
As $|M(G^\s)M(G^\s)^\top| \le M(G^+)M(G^+)^\top$, by the Perron-Frobenius theorem for nonnegative matrices (or the matrix version of Theorem \ref{eh} in Section 4),
$x$ contains no zero entries, and
$$ S^{-1}M(G^\s)M(G^\s)^\top S= M(G^+)M(G^+)^\top,$$
where $S=\diag\{x_e / |x_e|: e \in E(G)\}$, a signature matrix.
For any two edges $e, e' \in E(G)$,
$$ (S^{-1}M(G^\s)M(G^\s)^\top S)_{e,e'}= (M(G^+)M(G^+)^\top)_{e,e'}.$$
Noting that $S^{-1}=S$, we have
$$ \sum_{v \in e \cap e'} (S_e \s(e,v)) \cdot (S_{e'} \s(e',v))=\sum_{v \in e \cap e'} 1.$$
So, for all $v \in e \cap e'$,
$$ S_e \s(e,v)= S_{e'} \s(e',v).$$
Therefore, for any given vertex $v$, and for all edges $e$ incident to $v$,
$$ (SM(G^\s))_{e,v}=S_e \s(e,v),$$
which is a constant denoted by $\ell_v \in \{-1,1\}$.
Let $T=\diag\{\ell_v^{-1}: v \in V(G)\}$.
Then $SM(G^\s)T=M(G^+)$, and the result follows by Lemma \ref{veswt}.
\end{proof}

\begin{remark}
As $|M(G^\s)M(G^\s)^\top| \le M(G^+)M(G^+)^\top$,  $\lamax(G^+)$ is exactly the maximum singular value of $G^\s$ in Lemma \ref{ins}.
\end{remark}

By Lemma \ref{ins}, we easily get the following corollary.

\begin{cor} \label{LM} Let $G$ be a connected hypergraph.
Oriented hypergraph $G^\s$ is switching equivalent to $G^+$ if and only if  $\rho(L(G^+))$ is an eigenvalue of $L(G^\s)$.
\end{cor}

\begin{cor} \label{a} Let $G$ be a connected hypergraph.
Oriented hypergraph $G^\s$ is switching equivalent to $G^+$ if and only if $\rho(A(G^+))$ is an eigenvalue of $A(G^\s)$.
\end{cor}

\begin{proof}
Necessity. By Lemma \ref{veswt}, we have $M(G^+)=S_F M(G^\s) S_Z$ for some signature matrices $S_F$ and $S_Z$.
So $$L(G^+)=M(G^+)^\top M(G^+)=S_Z M(G^\s)^\top M(G^\s) S_Z=S_Z L(G^\s)S_Z.$$
We will get $D(G)+A(G^+)=S_Z(D(G)+A(G^\s))S_Z$, and hence $A(G^+)=S_Z^{-1} A(G^\s)S_Z$ as $S_Z=S_Z^{-1}$.
The necessity follows.

Sufficiency. Let $x$ be a real eigenvector of $A(G^\s)$ associated with $\rho(A(G^+))$.
Noting that $|A(G^\s)| \le A(G^+)$, by Perron-Frobenius theorem, $x$ contains no zero entries, and $S^{-1} A(G^\s) S =A(G^+)$, where
$S=\diag\{x_v/|x_v|: v \in V(G)\}$.
We have
$$S^{-1} L(G^\s)S=S^{-1} (D(G)+A(G^\s))S=D(G)+A(G^+)=L(G^+).$$
So $\rho(L(G^+))$ is an eigenvalue of $L(G^\s)$.
The sufficiency follows by Corollary \ref{LM}.
\end{proof}

Combining Theorem \ref{eqs}, Lemma \ref{ins}, Corollary \ref{LM} and Corollary \ref{a}, we get the main result of this section,
which provides a spectral characterization for incidence balanced oriented hypergraphs.

\begin{theorem}
Let $G^\s$ be a connected oriented hypergraph.
 Then the following statements are equivalent.

{\rm (1)} $G^\s$ is incidence balanced.

{\rm (2)} $\lamax(M(G^+))$ is a singular value of $M(G^\s)$.

{\rm (3)} $\rho(L(G^+))$ is an eigenvalue of $L(G^\s)$.

{\rm (4)} $\rho(A(G^+))$ is an eigenvalue of $A(G^\s)$.
\end{theorem}

\section{Tensor for signed hypergraphs}

Let $G$ be a connected $k$-uniform hypergraph, where $k$ is even.
Let $G^\s$ be an oriented hypergraph, and let $\G G^\s$ be the signed hypergraph induced by $G^\s$.
Clearly, a vertex switching on $G^\s$ naturally  gives rise to a vertex switching on $\G G^\s$,
   and any edge switching on $G^\s$ preserves the signs of edges of $\G G^\s$ as $k$ is even.
So,  if $G^\tau$ and $G^\s$ are switching equivalent in the setting of oriented hypergraphs,
then $\G G^\s$ and $\G G^\tau$ are switching equivalent in the setting of signed hypergraphs. However, the converse does not holds.

In this section we will discuss the signed hypergraphs induced by incidence balanced oriented hypergraphs.
We firstly introduce some knowledge about tensors for preparation.

\subsection{Tensors and eigenvalues}

A real {\it tensor} $\A=(a_{i_{1} i_2 \ldots i_{k}})$ of order $k$ and dimension $n$ is a
  multi-dimensional array with real entries $a_{i_{1}i_2\ldots i_{k}}$
  for all $i_{j}\in [n]$ and $j\in [k]$.
The tensor $\A$ is called {\it symmetric} if its entries are invariant under any permutation of their indices.

 Given a vector $x\in \mathbb{C}^{n}$, $\A x^{k} \in \mathbb{C}$ and $\A x^{k-1} \in \mathbb{C}^n$, which are defined as
\begin{align*}
\A x^{k} & =\sum_{i_1,i_{2},\ldots,i_{k}\in [n]}a_{i_1i_{2}\ldots i_{k}}x_{i_1}x_{i_{2}}\cdots x_{i_k},\\
(\A x^{k-1})_i & =\sum_{i_{2},\ldots,i_{k}\in [n]}a_{ii_{2}\ldots i_{k}}x_{i_{2}}\cdots x_{i_k}, i \in [n].
\end{align*}
In 2005, Lim \cite{Lim} and Qi \cite{Qi2} independently introduced the eigenvalues of tensors.

\begin{definition}\cite{Lim,Qi2}
Let $\A$ be a real tensor of order $k$ and dimension $n$.
For some $\lambda \in \mathbb{C}$, if the polynomial system $\A x^{k-1}=\lambda x^{[k-1]}$ has a solution $x \in \mathbb{C}^{n}\backslash \{0\}$,
then $\lambda $ is called an \emph{eigenvalue} of $\A$ and $x$ is an \emph{eigenvector} of $\A$ associated with $\lambda$,
where $x^{[k-1]}:=(x_1^{k-1}, x_2^{k-1},\ldots,x_n^{k-1})$.
\end{definition}

In the above definition,
 if $x$ is a real eigenvector of $\A$, surely the corresponding eigenvalue $\lambda$ is real.
In this case, $\lambda$ is called an {\it H-eigenvalue} of $\A$. Furthermore, an eigenvalue is called {\it H$^+$-eigenvalue}
({\it H$^{++}$-eigenvalue}, respectively) if it is associated with a nonnegative (positive, respectively) eigenvector.
The spectral radius of $\A$ is defined as the maximum modulus of the eigenvalues of $\A$, denoted by $\rho(\A)$.

Chang et al. \cite{CPZ} generalize the Perron-Frobenius theorem from nonnegative matrices to nonnegative tensors.
Friedland et al. \cite{FGH} get some further results for weakly irreducible nonnegative tensors.

\begin{theorem} \cite{FGH} \label{Perron}
If $\A$ is a weakly nonnegative irreducible tensor, then

\begin{enumerate}

\item $\rho(\A)$ is an H$^{++}$-eigenvalue of $\A$.

\item If $\lambda$ is an eigenvalue of $\A$ associated with a positive eigenvector, then $\lambda=\rho(\A)$.

\end{enumerate}
\end{theorem}

Let $\mathcal{A}=(a_{i_{1}}\ldots _{i_{k}})$ be a tensor of order $k$ and dimension $n$,
and let $P=\diag \{p_i: i \in [n]\}$ and $Q=\diag \{q_i: i \in [n]\}$ be two diagonal matrices.
The product $P\mathcal{A}Q$ is defined as a tensor of order $k$ and dimension $n$ \cite{Shao}, where
$$(P \A Q)_{i_1\ldots i_k }=p_{i_1}a_{i_{1}}\ldots _{i_{k}}q_{i_2}\cdots q_{i_k}.$$
If $P=Q^{-(k-1)}$, then $\A$ and $P \A Q$ are called {\it diagonal similar}, and have the same spectrum \cite{Shao}.
 Yang and Yang \cite{YY1, YY2, YY3} get further results for Perron-Frobenius theorem.

\begin{theorem}{\em \cite{YY3}}\label{eh}
Let $\mathcal{A}$ and $\mathcal{B}$ be $k$-th order $n$-dimensional tensors with $|\mathcal{B}|\leq \mathcal{A}$.\\
Then

\begin{enumerate}

\item $\rho(\mathcal{B})\leq \rho(\mathcal{A})$.

\item If $\mathcal{A}$ is weakly irreducible and $\rho(\mathcal{B})=\rho(\mathcal{A})$, where $\lambda=\rho(\mathcal{A})e^{i\theta}$ is an eigenvalue of $\mathcal{B}$ corresponding to an eigenvector $y$, then $y=(y_{1},\cdots,y_{n})$ contains no zero entries, and $\mathcal{A}=e^{-i\theta}D^{-(k-1)}\mathcal{B}D$,
where $D=\diag \{{y_{i}}/{|y_{i}|}: i \in [n]\}$.
\end{enumerate}

\end{theorem}

\subsection{Spectral characterization of the signed hypergraphs}
Let $G$ be a $k$-uniform hypergraph with vertex set $[n]$.
The {\it adjacency tensor} \cite{YYQ} of a signed hypergraph $\G=(G,\g)$ is defined as $\A(\G)=(a^\g_{i_1 \ldots a_k})$, a $k$-th order and $n$-dimensional symmetric tensor, where
\[a^\g_{i_{1}\ldots i_{k}}=\left\{
 \begin{array}{ll}
\frac{\g(e)}{(k-1)!}, &  \mbox{if~} e=\{i_{1},\ldots,i_{k}\} \in E(G);\\
  0, & \mbox{otherwise}.
  \end{array}\right.
\]
Let $\D(G)$ be a  $k$-th order and $n$-dimensional diagonal tensor, called {\it degree tensor} of $G$, where $d_{i\ldots i}$ is the degree of the vertex $i$ for all $i \in [n]$.
The {\it Laplacian tensor} of $\G$ is defined to be $ \L(\G)=\D(G)+\A(\G)$.

Note that for the signed hypergraph $\G G^\s$ induced by oriented hypergraph $G^\s$, the sign of each edge $e$ is given by $\sgn_\s e$.
If $k$ is even, by the definition the adjacency tensor $\A(\G G^+)=-\A(G)$, where $\A(G)$ is the adjacency tensor of the hypergraph $G$ introduced in \cite{CD},
and the Laplacian tensor $\L(\G G^+)=\D(G)-\A(G)$ which is exactly the Laplacian tensor $\L(G)$ of $G$ introduced in \cite{Qi}.
If $k=2$, then $\A(G)$ and $\L(G)$ are the usual adjacency matrix and Laplacian matrix of a simple graph $G$ respectively.

For a uniform hypergraph $G$, it is proved that $\A(G)$ is weakly irreducible if and only if $G$ is connected \cite{PZ,YY3}.
The result also holds for $\L(G)$, and $\A(\G), \L(\G)$ of signed hypergraphs $\G$.

\begin{lemma}\label{SS} Let $G^\s$ and $G^{\tau}$ be two $k$-uniform oriented hypergraphs.
 Signed hypergraph $\G G^{\tau}$ is switching equivalent to $\G G^\s$  if and only if $\A(\G G^{\tau})=S \A(\G G^\s) S$ for some signature matrix $S$.
\end{lemma}

\begin{proof} Necessity.
Suppose $\G G^{\tau}$ is obtained from $\G G^\s$ by making a sequence of vertex switchings $S_v$ in turn at each vertex $v \in W$.
Recall the signature matrix $S_W:= \prod_{v \in W}S_v$ defined in Section 2.2.
For any edge $e=\{i_1,  \ldots, i_k\}$,
\begin{equation}
\A(\G G^{\tau})_{i_1 \ldots i_k}=\A(\G G^\s)_{i_1 \ldots i_k} \prod_{i_j \in e}S_{i_ji_j}=S_{i_1i_1}\A(\G G^\s)_{i_1 i_2 \ldots i_k}S_{i_2i_2}\ldots S_{i_ki_k}, \label{ds}
 \end{equation}
 which is equivalent to $\A(\G G^{\tau})=S \A(\G G^\s) S$ by the definition.

Sufficiency.  Suppose  $\A(\G G^{\tau})=S \A(\G G^\s) S$ for some signature matrix $S$,
where $S_{uu}=-1$ for all $u \in W$, and  $S_{uu}=1$ otherwise.
It follows from (\ref{ds})
that making vertex switchings for each vertex $v \in W$ on $\G G^\s$,
we obtain a resultant hypergraph whose adjacency tensor is the same as $\A(\G G^{\tau})$.
 Note that two signed hypergraphs with the same underlying hypergraph
     are same if and only if their adjacency tensors are same (or the signs of corresponding edges are same).
So $\G G^{\tau}$ can be obtained from $\G G^\s$ by a sequence of vertex switchings at each vertex of $W$.
The sufficiency follows.
\end{proof}

\begin{theorem}\label{minusR} Let $G^\s$ be a connected $k$-uniform oriented hypergraph, where $k$ is even.
Signed hypergraph $\G G^{\s}$ is switching equivalent to $\G G^+$ if and only if
$-\rho(\A(G))$ is an H-eigenvalue of $\A(\G G^\s)$.
\end{theorem}

\begin{proof} Necessity. Suppose $\G G^{\s}$ is switching equivalent to $\G G^+$.
It follows from Lemma \ref{SS} that for some signature matrix $S$,
$$\A(\G G^+)=-\A(G)=S \A(\G G^\s) S=S^{-(k-1)}\A(\G G^\s)S, $$
where the third equality holds as $S^k=I$ for even $k$.
By Theorem \ref{Perron}, $\rho(\A(G))$ is an eigenvalue of $\A(G)$ associated with a positive eigenvector $y$.
So $-\rho(\A(G))$ is an eigenvalue of $\A(\G G^{\sigma})$ associated with a real eigenvector $Sy$.

Sufficiency. Let $y$ be a real eigenvector of $A(\G G^{\sigma})$ associated with $-\rho(\A(G))$, namely
$$-\mathcal{A}(\G G^{\sigma})y^{k-1}=\rho(\A(G))y^{[k-1]}.$$
Note that $|-\mathcal{A}(\G G^{\sigma})|=\A(G)$. It follows from Theorem \ref{eh} that $y$ contains no zero entry and
\begin{align*}\mathcal{A}(G)=D_{y}^{-(k-1)} (-\mathcal{A}(\G G^{\sigma}))D_{y},
\end{align*}
where $D_{y}=\diag \{{y_v}/{|y_v|}: v \in V(G)\}$.
Note that $D_{y}$ is a signature matrix as $y$ is real.
As $k$ is even,  we have
\begin{align*}
\label{eq}
\mathcal{A}(\G G^+)=-\mathcal{A}(G)=D_{y}^{-(k-1)}\mathcal{A}(\G G^{\sigma})D_{y}=D_{y}\mathcal{A}(\G G^{\sigma})D_{y}.
\end{align*}
The sufficiency follows by Lemma \ref{SS}.
\end{proof}

\begin{lemma}\label{even42}
Let $G^\s$ be a connected $k$-uniform oriented hypergraph, where $k \equiv 2 \mod 4$.
Then the  following are equivalent.

\begin{enumerate}
\item $-\rho(\A(G))$ is an H-eigenvalue of $\A(\G G^\s)$.

\item $-\rho(\A(G))$ is an eigenvalue of $\A(\G G^\s)$.

\end{enumerate}

\end{lemma}

\begin{proof}
It suffices to prove (2) $\Rightarrow$ (1).
Let $y$ be an eigenvector of $\A(\G G^\s)$ associated with eigenvalue $-\rho(\A(G))$.
By a similar discussion to the sufficiency in the proof of Theorem \ref{minusR},
we have that $y$ contains no zero entry and
\begin{equation}\label{nonH1}
\mathcal{A}(G)=D_{y}^{-(k-1)} (-\mathcal{A}(\G G^{\sigma}))D_{y},
\end{equation}
where $D_{y}=\diag\{{y_v}/{|y_v|}:  v \in V(G)\}=:\diag\{d_v: v\in V\}$.
To complete the proof, it is sufficient to find a signature matrix $T$ such that
\begin{equation}\label{nonH2}
-\mathcal{A}(G) =T^{-(k-1)}\mathcal{A}(\G G^{\sigma})T.
\end{equation}

For any edge $e=\{i_1, \ldots, i_k\} \in E(G)$, by Eq. (\ref{nonH1}), we have
$$
-a_{i_1\ldots i_k}=d_{i_1}^{-(k-1)} (-1)^{k-1}\sigma^e a_{i_1\ldots i_k}d_{i_2} \ldots d_{i_k}=-\frac{d^e}{d_{i_1}^k} \sigma^e a_{i_1 \ldots i_k},
$$
where $d^e:=\prod_{i_j \in e} d_{i_j}$.
So, we have
 \begin{equation}\label{dvC}
\s^e d^e=d_v^k, \hbox{~for any~} v \in e.
\end{equation}

 For a specified vertex $v$, we can normalize the eigenvector $y$ such that
 $d_v=\frac{y_v}{|y_v|}=1$.
 As $G$ is connected, by Eq. (\ref{dvC}) we have $d_u^k=1$ for all vertices $u \in V(G)$, and $\s^e d^e=1$ for all edges $e \in E(G)$.

Write $d_u=e^{\frac{\im 2 \pi}{k} \ell_u}$ for  each vertex $u \in V(G)$,  where $\mathbf{i}=\sqrt{-1}$, $\ell_u \in [k]$.
If $\s^e=1$, then $d^e=1$, namely,
$$ d^e=e^{\frac{\im 2 \pi}{k} \sum_{u \in e} \ell_u}=1,$$
which implies that
\begin{equation}\label{one} \sum_{u \in e} \ell_u \equiv 0 \mod k.\end{equation}
If $\s^e=-1$, then $d^e=-1$, namely,
$$ d^e=e^{\frac{\im 2 \pi}{k} \sum_{u \in e} \ell_u}=-1,$$
which implies that
\begin{equation}\label{minusone} \sum_{u \in e} \ell_u \equiv \frac{k}{2} \mod k.\end{equation}

Denote $t_u:=(-1)^{\ell_u}$ for  each vertex $u \in V(G)$, and $t^e:=\prod_{u \in e}t_u$.
As $k \equiv 2 \mod 4$, if Eq. (\ref{one}) holds, then
$ \sum_{u \in e} \ell_u \equiv 0 \mod 2$, and $t^e=d^e=1$.
If  Eq. (\ref{minusone}) holds, then
$ \sum_{u \in e} \ell_u \equiv 1 \mod 2$, and $t^e=d^e=-1$.
So, in either case we have $t^e=d^e$ for each edge $e \in E(G)$.
Noting that $t_v^k=d_v^k=1$, we have
$$\s^e t^e=t_v^k$$
for any edge $e$ and any vertex $v \in e$.
Let $T=\diag\{t_v: v \in V(G)\}$, a signature matrix.
It is easy to verify $-\A(G)=T^{-(k-1)}\A(\G G^\s)T$, and hence $-\rho(\A(G))$ is an H-eigenvalue of $\A(G^\s)$.
\end{proof}

\begin{remark} \label{0mod4}
By Lemma \ref{even42},  we can replace the H-eigenvalue by eigenvalue in Theorem \ref{minusR} when $k \equiv 2  \mod 4$.
However, it does not hold when $k \equiv 0 \mod 4$ as evidenced by the following example.
\end{remark}

Before we illustrate Remark \ref{0mod4} by an example, we need to introduce some related knowledge.
An even uniform hypergraph $G$ is called {\it odd-bipartite} (or {\it odd-transversal}) if $V(G)$ has a bipartition $\{V_1, V_2\}$
such that each edge intersections $V_1$ (and $V_2$) in an odd number of vertices; otherwise, $G$ is called {\it non-odd-bipartite}.

\begin{lemma} \cite{FY,SSW} \label{odd}
	Let $G$ be a connected $k$-uniform hypergraph. Then $-\rho(\A(G))$ is an H-eigenvalue of $\A(G)$ if and only if $k$ is even and $G$ is odd-bipartite.
\end{lemma}

\begin{example} \label{ex}
Let $G^{\sigma}$ be a $4$-uniform oriented hypergraph with vertex set $V=[6]$,
and edge set $E= \{e_1=\{1,2,3,4\}, e_2=\{{1,2,5,6\}}, e_3=\{3,4,5,6\}\}$, and an incidence orientation $\s$
 which gives all edge-vertex incidences positive orientations except $(e_1,2), (e_2,6), (e_3,3)$, as shown in Fig. \ref{nonbi}.
By using the terminology in \cite{KF}, the hypergraph $G$ can also be denoted by $C_3^{4,2}$,
   which is obtained from a triangle $C_3$ (as a simple graph) by blowing up each vertex into a $2$-set.
It is known  from \cite[Theorem 2.3, Lemma 3.12]{KF} that $C_3^{4,2}$ is non-odd-bipartite as $C_3$ is non-bipartite, and $\rho(\A(C_3^{4,2}))=\rho(\A(C_3))=2$.

Observe that $\A(\G G^\s)=\A(G)$, and $-2$ is an eigenvalue of $\A(G)$ associated with an eigenvector $x=(\im, 1, \im, 1, \im, 1)$.
But $-2$ is not an H-eigenvalue of $\A(G)$; otherwise $G$ is odd-bipartite by Lemma \ref{odd} which yields a contradiction.

We assert that $\G G^{\s}$ cannot be switching equivalent to $\G G^+$.
Otherwise by Lemma \ref{SS} $\A(\G G^\s)=S \A(\G G^+)S$ for some signature matrix $S$.
Then we have
$$\A(G)=\A(\G G^\s)=S \A(\G G^+)S=S^{-3} (-\A(G))S,$$
which implies that $-\rho(\A(G))$ is an H-eigenvalue of $\A(G)$, a contradiction.
\end{example}

\begin{figure}
	\centering
		\setlength{\unitlength}{1bp}%
	\begin{picture}(90, 106)(0,0)
	\put(0,0){\includegraphics[scale=.8]{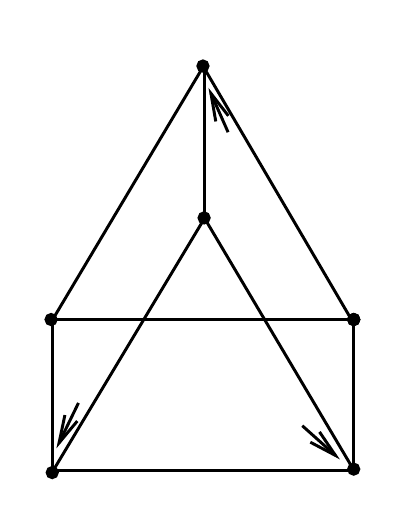}}
	\put(4.00,42.00){\fontsize{10.69}{12.83}\selectfont 1}
	\put(4.00,6.00){\fontsize{10.69}{12.83}\selectfont 2}
	\put(50.00,102.00){\fontsize{10.69}{12.83}\selectfont 3}
	\put(50.00,68.00){\fontsize{10.69}{12.83}\selectfont 4}
	\put(85.60,42.00){\fontsize{10.69}{12.83}\selectfont 5}
	\put(85.60,6.00){\fontsize{10.69}{12.83}\selectfont 6}
	\put(28.00,60.00){\fontsize{11.76}{14.11}\selectfont $e_1$}
	\put(42.00,25.00){\fontsize{12.93}{15.52}\selectfont $e_2$}
	\put(56.00,60.00){\fontsize{11.76}{14.11}\selectfont $e_3$}
	\end{picture}%
		\caption{A $4$-uniform oriented hypergraph $G^\s$} \label{nonbi}
	\end{figure}
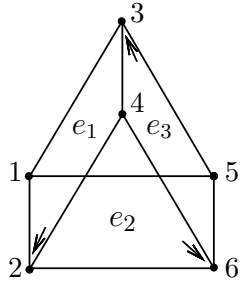

Note that
$$ \L(\G G^\s)x^k= \sum_{e \in E}(x_e^k + k (\sgn_\s e) x^e),$$
where $x_e^k:=\sum_{v \in e}x_v^k$, and $x^e:=\prod_{v \in e}x_v$.
If $k$ is even, by AM-GM inequality, for any real vector $x$,
$$x_e^k + k (\sgn_\s e) x^e \ge 0,$$
implying that $\L(\G G^\s)$ is positive semidefinite with all H-eigenvalues nonnegative.

At the end of this paper, we get our main result of this section complementary to Lemma \ref{SS} and Theorem \ref{minusR},
which provides a spectral characterization for signed hypergraphs induced by the incidence balanced oriented hypergraphs.

\begin{theorem} \label{Lazero}
Let $G^\s$ be a connected $k$-uniform oriented hypergraph, where $k$ is even. Then the following are equivalent.

\begin{enumerate}
	\item $\G G^{\s}$ is switching equivalent to $\G G^+$.
	\item $S \A(\G G^{\s}) S=-\A(G)$ for some signature matrix $S$.
	\item $-\rho(\A(G))$ is an H-eigenvalue of $\A(\G G^\s)$.
	\item $S \L(\G G^{\s}) S=\L(G)$ for some signature matrix $S$.
	\item $0$ is an H-eigenvalues of $\L(\G G^\s)$.
	\item $G^\s$ has an bipartition $\{V^+,V^-\}$ such that
	each positive edge (respectively negative edge) intersects $V^-$ in an odd number (respectively, even number) of vertices.
\end{enumerate}	
\end{theorem}

\begin{proof} It is clear that $(1)\Leftrightarrow (2) \Leftrightarrow (3)$ by  Lemma \ref{SS} and Theorem \ref{minusR}.
	We note that
	$$S \L(\G G^{\s}) S=S(\D(G)+\A(\G G^\s))S=\D(G)+ S(\A(\G G^\s))S.$$
	So the equivalence of (2) and (4) follows.
	
(4)	$\Rightarrow$ (5). It follows from (4) that
$\L(G) = S \L(\G G^{\s}) S=S^{-(k-1)} \L(\G G^{\s}) S$ as $S^k=I$ for even $k$.
So $\L(\G G^{\s})$ and $\L(G)$ are diagonal similar, and hence have the same spectrum.
Since $\L(G)$ always has a zero eigenvalue associated with the all-ones vector as an eigenvector, the statement (5) follows.

(5)	$\Rightarrow$ (6). If $0$ is an H-eigenvalue of $\L(\G G^{\s})$, then there exists a real eigenvector $x$ such that
$$ \L(\G G^\s)x^k= \sum_{e \in E}(x_e^k + k (\sgn_\s e) x^e)=0.$$
Noting that $k$ is even, so for each edge $e$,
$$x_e^k + k (\sgn_\s e) x^e = 0,$$
implying that for any two vertices $v, v'$ in $e$, $|x_v|=|x_{v'}|$ and $ (\sgn_\s e) \sgn x^e \le 0$.
As $x$ contains a nonzero entry, by the connectedness of $G$, $|x_v|$ is a non-zero constant for all $v$.
So we may assume $x_v \in \{-1,1\}$.

Let $V^+=\{v \in V: x_v =1\}$ and let $V^-=\{v \in V: x_v =-1\}$.
Then $V^+, V^-$ consist of a (possibly trivial) bipartition of $V(G)$.
So for each positive edge $e$ (with $\sgn_\s e>0$), $\sgn x^e=-1$, and hence $e$ intersects $V^-$ in an odd number of vertices.
Similarly, for each negative edge $e$ (with $\sgn_\s e<0$), $\sgn x^e=1$ and hence $e$ intersects $V^-$ in an even number of vertices.

(6)	$\Rightarrow$ (4). Define $S$ to be a signature matrix such that $S_{vv} =1$ if $v \in V^+$ and $S_{vv}=-1$ otherwise.
It is easy to verify that $S\L(\G G^\s)S=\L(G)$.
\end{proof}

\bibliographystyle{amsplain}

\end{document}